\documentclass{article}%
\usepackage{mathrsfs}
\usepackage{amsmath}
\usepackage{amsfonts}
\usepackage{amssymb}
\usepackage{graphicx}%
\providecommand{\U}[1]{\protect \rule{.1in}{.1in}}

\newtheorem{theorem}{Theorem}[section]

\newtheorem{definition}[theorem]{Definition}

\newtheorem{lemma}[theorem]{Lemma}

\newtheorem{proposition}[theorem]{Proposition}
\newtheorem{remark}[theorem]{Remark}

\newenvironment{proof}[1][Proof]{\noindent \textbf{#1.} }{\  $\Box$}
\begin{document}

\title{On the Representation Theorem of $G$-Expectations and Paths of $G$--Brownian Motion}
\author{Mingshang HU and Shige PENG\thanks{The author thanks the partial support from The National Basic
Research Program of China (973 Program) grant No. 2007CB814900
(Financial Risk). }\\Institute of Mathematics\\Shandong
University\\250100, Jinan, China\\peng@sdu.edu.cn}

\maketitle

\begin{abstract}
We give a very simple and elementary proof of the existence of a
weakly compact family of probability measures $\{P_{\theta}:\theta
\in \Theta \}$ to represent an important sublinear expectation---
$G$-expectation $\mathbb{E}[\cdot]$. We also give a concrete
approximation of a bounded continuous function $X(\omega)$ by an
increasing sequence of cylinder functions $L_{ip}(\Omega)$ in order
to prove that $C_{b}(\Omega)$ belongs to the
$\mathbb{E}[|\cdot|]$-completion of the $L_{ip}(\Omega)$.

\end{abstract}

\medskip

\noindent{\footnotesize{\bf Keywords:\hspace{2mm} Probability and
distribution uncertainty,
$G$-normal distribution, $G$-Brownian motion, Continuous paths}}

 \vspace{2mm}
\baselineskip 15pt
\renewcommand{\baselinestretch}{1.10}
\parindent=16pt  \parskip=2mm
\rm\normalsize\rm

\section{Introduction}

Recently a new stochastic process called $G$-Brownian motion has
been introduced in \cite{Peng2006a, Peng2006b} under a framework of
sublinear expectation called $G$-expectation $\mathbb{E}$. From the
well-known representation theorem of sublinear expectation, a
$G$-expectation $\mathbb{E}$ can be represented by an upper
expectation: $\mathbb{E}[\cdot]=\sup_{\lambda \in \Lambda}E_{\lambda
}[\cdot]$, where $\{E_{\lambda}:\lambda \in \Lambda \}$ is a family
of finitely additive linear expectation (see [Huber], [Delb2] and
[P5]). In \cite{DHP} Denis, Hu and Peng have introduced a method of
optimal stochastic controls (see [DHP], Section 4.1) to construct a
weakly compact family of ($\sigma$-additive) probability measures
$\{P_{\theta}:\theta \in \Theta \}$ such that
\[
\mathbb{E}[X]=\sup_{\lambda \in \Lambda}\int_{\Omega}X(\omega)dP_{\lambda}.
\]
where $\Omega$ is the space of continuous paths. Since the representation
$\sup_{\lambda \in \Lambda}E_{\lambda}[\cdot]$ is very elementary---only
Hahn-Banach theorem is involved--- a nature question is: Can we use the family
$\{E_{\lambda}:\lambda \in \Lambda \}$ to find $\{P_{\theta}:\theta \in \Theta \}$
instead of passing the above mentioned long proof by using sophisticate
stochastic control theory?

In this paper we give an affirmative answer to this question. Our
method can be regarded as a combination and extension of the
original Brownian motion construction approach of Kolmogorov and the
Lipschitz cylinder functions $L_{ip}(\Omega)$ (see Section 2 for its
definition) introduced in \cite{Peng2005} and \cite{Peng2006a}. This
permits to give a much simpler proof involving only elementary
results of probability theory. The proof is short but the importance
is obvious since it involves the foundation of the theory of
$G$-Brownian motion and the related stochastic calculus.

In this paper, we also give a concrete approximation of a bounded continuous
function $X(\omega)$ by an increasing sequence of bounded and Lipschitz
functions $L_{ip}(\Omega)$ in order to prove that $C_{b}(\Omega)$ belongs to
the $\mathbb{E}[|\cdot|]$-completion of $L_{ip}(\Omega)$.

This paper is organized as follows: in Section 2, we use Hahn-Banach theorem
to prove representation theorem of sublinear expectation. In Section 3, we
find a weakly compact family of probability measures to represent
$G$-expectation. In Section 4, we prove that every bounded continuous function
belongs to the $\mathbb{E}[|\cdot|]$-completion of $L_{ip}(\Omega)$.

\section{Basic settings of $G$-Brownian motion and $G$-expectation}

We present some preliminaries in the theory of sublinear
expectations and the related $G$-Brownian motions. More details of
this section can be found in [P5] and [P2008].

\begin{definition}
{\label{Def-1} { Let }}$\Omega$ be a given set and let $\mathscr{H}$
be a linear space of real valued functions defined on $\Omega$ with
$c\in \mathscr{H}$ for all constants $c$. $\mathscr{H}$ is
considered as the space of our \textquotedblleft random
variables\textquotedblright. {{A \textbf{nonlinear expectation
}$\mathbb{\hat{E}}$ on $\mathscr{H}$ is a functional
$\mathbb{\hat{E}}:\mathscr{H}\mapsto \mathbb{R}$ satisfying the
following properties: for all $X,Y\in \mathscr{H}$, we have\newline%
\  \  \newline \textbf{(a) Monotonicity:} \  \  \  \  \  \  \  \  \  \  \  \  \ If $X\geq Y$
then $\mathbb{\hat{E}}[X]\geq \mathbb{\hat{E}}[Y].$\newline \textbf{(b) Constant
preserving: \  \ }\  \ $\mathbb{\hat{E}}[c]=c$.\newline The triple }}%
$(\Omega,\mathscr{H},\mathbb{\hat{E}})$ is called a nonlinear
expectation space (compare with a probability space
$(\Omega,\mathscr{F},\mathbb{P})$). {{In this paper we are mainly
concerned with sublinear expectation where the expectation
$\mathbb{\hat{E}}$ satisfies also}}\newline{{\textbf{(c)}
\textbf{Sub-additivity: \  \  \  \ }}}\  \  \  \  \  \  \  \ $\mathbb{\hat{E}%
}[X]-\mathbb{\hat{E}}[Y]\leq \mathbb{\hat{E}}[X-Y].$\newline{{\textbf{(d)
Positive homogeneity: } \ $\mathbb{\hat{E}}[\lambda X]=\lambda \mathbb{\hat{E}%
}[X]$,$\  \  \forall \lambda \geq0$.\newline}}\newline If only
\textbf{(c) }and \textbf{(d) }are satisfied, $\mathbb{\hat{E}}$ is
called a sublinear functional.
\end{definition}

The following representation theorem for sublinear expectations is
very useful (see \cite{NewCLT} for the proof).

\begin{lemma}
\label{le0}\bigskip Let {{$\mathbb{\hat{E}}$}} be a sublinear
functional defined on a linear space $\mathscr{H}$, i.e.,
\textbf{(c) }and \textbf{(d)} hold for $\mathbb{\hat{E}}$. Then
there exists a family $\mathscr{Q}=\{E_{\theta}:\theta \in \Theta
\}$ of linear functionals defined on $\mathscr{H}$ such that
\[
{{\mathbb{\hat{E}}}}[X]:=\sup_{\theta \in \Theta}E_{\theta}[X],\  \  \text{for }%
X\in \mathscr{H}.
\]
and such that, for each $X\in \mathscr{H}$, there exists a $\theta
\in \Theta$ such that {{$\mathbb{\hat{E}}$}}$[X]:=E_{\theta}[X]$. If
we assume moreover that $\mathbb{\hat{E}}$ is a sublinear functional
defined on a linear space $\mathscr{H}$ of functions on $\Omega$
such that \textbf{(a) }holds (resp.\textbf{ (a), (b)} hold) for
$\mathbb{\hat{E}}$, then \textbf{(a) }also holds \ (resp.\textbf{
(a), (b)} hold) for $E_{\theta}$, $\theta \in \Theta$.{ }
\end{lemma}

For a given positive integer $n$ we will denote by $(x,y)$ the
scalar product of $x$, $y\in \mathbb{R}^{n}$ and by $\left \vert
x\right \vert =(x,x)^{1/2}$ the Euclidean norm of $x$. We often
consider a nonlinear expectation space
$(\Omega,\mathscr{H},\mathbb{\hat{E})}$ such that $X_{1}$,$\cdots$,$X_{n}%
\in \mathscr{H}$ implies $\varphi(X_{1},\cdots,X_{n})\in
\mathscr{H}$ for each $\varphi \in C_{l.Lip}(\mathbb{R}^{n})$, where
$C_{l.Lip}(\mathbb{R}^{n})$ is the space of real continuous
functions defined on $\mathbb{R}^{n}$ such that
\[
|\varphi(x)-\varphi(y)|\leq C(1+|x|^{k}+|y|^{k})|x-y|,\  \  \forall
x,y\in \mathbb{R}^{n},
\]
where $k$ depends only on $\varphi$.

We recall some important notions of nonlinear expectations
distributions (see [P5,Peng2008]):

\begin{definition}
Let $X_{1}$ and $X_{2}$ be two $n$--dimensional random vectors defined
respectively in {sublinear expectation spaces }$(\Omega_{1},\mathscr{H}%
_{1},\mathbb{\hat{E}}_{1})${ and
}$(\Omega_{2},\mathscr{H}_{2},\mathbb{\hat {E}}_{2})$. They are
called identically distributed, denoted by $X_{1}\sim X_{2}$, if
\[
\mathbb{\hat{E}}_{1}[\varphi(X_{1})]=\mathbb{\hat{E}}_{2}[\varphi
(X_{2})],\  \  \  \forall \varphi \in C_{l.Lip}(\mathbb{R}^{n}).
\]

\end{definition}

\begin{definition}
In a sublinear expectation space
$(\Omega,\mathscr{H},\mathbb{\hat{E}})$ a random vector
$Y=(Y_{1},\cdots,Y_{n})$, $Y_{i}\in \mathscr{H}$ is said to be
independent to another random vector $X=(X_{1},\cdots,X_{m})$, $X_{i}%
\in \mathscr{H}$ under $\mathbb{\hat{E}}[\cdot]$ if for each test
function $\varphi \in C_{l.Lip}(\mathbb{R}^{m}\times
\mathbb{R}^{n})$ we have
\[
\mathbb{\hat{E}}[\varphi(X,Y)]=\mathbb{\hat{E}}[\mathbb{\hat{E}}%
[\varphi(x,Y)]_{x=X}].
\]
$\bar{X}=(\bar{X}_{1},\cdots,\bar{X}_{m})$ is said to be an
independent copy of $X$ if $\bar{X}\sim X$ and $\bar{X}$ is
independent to $X$.
\end{definition}

\begin{definition}
(\textbf{$G$-normal distribution}) \label{Def-Gnormal} A
$d$-dimensional random vector $X=(X_{1},\cdots,X_{d})$ in a
sublinear expectation space $(\Omega,\mathscr{H},\mathbb{\hat{E}})$
is called $G$-normal distributed if for each $a\,$, $b\geq0$ we have
\begin{equation}
aX+b\bar{X}\overset{d}{=}\sqrt{a^{2}+b^{2}}X,\  \label{srt-a2b2}%
\end{equation}
where $\bar{X}$ is an independent copy of $X$. Here the letter $G$ denotes the
function
\[
G(A):=\frac{1}{2}\mathbb{\hat{E}}[(AX,X)]:\mathbb{S}_{d}\mapsto
\mathbb{R}.
\]

\end{definition}

\begin{remark}
It is easy to prove that the function $G$ is a monotonic and sublinear
function:%
\[
\left \{
\begin{array}
[c]{rl}%
G(A+\bar{A}) & \leq G(A)+G(\bar{A}),\\
G(\lambda A) & =\lambda G(A),\  \  \forall \lambda \geq0,\\
G(A) & \geq G(\bar{A}),\ if\  \ A\geq \bar{A}.
\end{array}
\right.
\]
From Lemma \ref{le0}, there exists a (bounded) subset $\Sigma \subset
\mathbb{S}_{d}$ such that $\gamma \geq0$ for each $\gamma \in \Sigma$ and%
\[
G(A)=\frac{1}{2}\sup_{\gamma \in \Sigma}\mathrm{tr}[A\gamma],\quad
A\in\mathbb{S}_{d}.
\]
We often denote $X\sim \mathscr{N}(0,\Sigma)$. In \cite{Peng2006a,
Peng2006b, Peng2007b, NewCLT} it is proved that for each given
monotonic and sublinear function $G$ defined on $\mathbb{S}_{d}$
there exists a random vector in some sublinear expectation space
$(\Omega,\mathscr{H},\mathbb{\hat{E}})$ such that $X\sim
\mathscr{N}(0,\Sigma )$, namely, $X$ is $G$-normal distributed. It
is also proved in Peng [P5,Peng2008] that,
for each $\mathbf{a}\in \mathbb{R}^{d}$ and $p\in \lbrack1,\infty)$%
\[
\mathbb{\hat{E}}[|\left(  \mathbf{a},X\right)  |^{p}]=\frac{1}{\sqrt
{2\pi
\sigma_{\mathbf{aa}^{T}}^{2}}}\int_{-\infty}^{\infty}|x|^{p}\exp
\left( \frac{-x^{2}}{2\sigma_{\mathbf{aa}^{T}}^{2}}\right)  dx,
\]
where $\sigma_{\mathbf{aa}^{T}}^{2}=2G(\mathbf{aa}^{T})$.
\end{remark}

\begin{definition}
\label{Def-3}(\textbf{\cite{Peng2006a} and \cite{Peng2007b}}) Let $G:\mathbb{S}%
_{d}\mapsto \mathbb{R}$ be a given monotonic and sublinear function.
A process $\{B_{t}(\omega)\}_{t\geq0}$ in a sublinear expectation
space $(\Omega ,\mathscr{H},\mathbb{\hat{E}})$ is called a
$G$\textbf{--Brownian motion} if
for each $n\in \mathbb{N}$ and $0\leq t_{1},\cdots,t_{n}<\infty$,$\ B_{t_{1}%
},\cdots,B_{t_{n}}\in \mathscr{H}$ and the following properties are
satisfied:
\newline \textsl{(i)} $B_{0}(\omega)=0$;\newline \textsl{(ii)} For each
$t,s\geq0$, the increment $B_{t+s}-B_{t}$ is independent to $(B_{t_{1}},B_{t_{2}}%
,\cdots,B_{t_{n}})$, for each $n\in \mathbb{N}$ and $0\leq t_{1}\leq
\cdots \leq t_{n}\leq t$; \newline \textsl{(iii)} $B_{t+s}-B_{t}\sim
\sqrt{s}X$, for $s,t\geq0,$ where $X$ is $G$-normal distributed.
\end{definition}


Let $\bar{\Omega}=(\mathbb{R}^{d})^{[0,\infty)}$ denote the space of
all $\mathbb{R}^{d}-$valued functions $(\bar{\omega}_{t})_{t\in
\mathbb{R}^{+}}$ and $\mathscr{B}(\bar{\Omega})$ denote the
$\sigma$-algebra generated by all
finite dimensional cylinder sets. Correspondingly, we denote by $\Omega=C_{0}%
^{d}(\mathbb{R}^{+})$ the space of all $\mathbb{R}^{d}-$valued continuous
functions $(\omega_{t})_{t\in \mathbb{R}^{+}}$, with $\omega_{0}=0$, equipped
with the distance
\[
\rho(\omega^{1},\omega^{2}):=\sum_{i=1}^{\infty}2^{-i}[(\max_{t\in \lbrack
0,i]}|\omega_{t}^{1}-\omega_{t}^{2}|)\wedge1].
\]
$\mathscr{B}(\Omega)$ denotes the $\sigma$-algebra generated by all open sets.
The corresponding canonical process $\bar{B}_{t}(\bar{\omega})=\bar{\omega
}_{t}$, (resp. $B_{t}(\omega)=\omega_{t}$) $t\in \lbrack0,\infty)$ for
$\bar{\omega}\in \bar{\Omega}$\ (resp.\ $\omega \in \Omega$). The spaces of
Lipschitzian cylinder functions on $\Omega$ and $\bar{\Omega}$ are denoted
respectively by
\[
L_{ip}(\bar{\Omega}):=\{ \varphi(\bar{B}_{t_{1}},\bar{B}_{t_{2}},\cdots,\bar
{B}_{t_{n}}):\forall n\geq1,t_{1},\cdots,t_{n}\in \lbrack0,\infty
),\forall \varphi \in C_{l.Lip}(\mathbb{R}^{d\times n})\},
\]%
\[
L_{ip}(\Omega):=\{ \varphi(B_{t_{1}},B_{t_{2}},\cdots,B_{t_{n}}):\forall
n\geq1,t_{1},\cdots,t_{n}\in \lbrack0,\infty),\forall \varphi \in C_{l.Lip}%
(\mathbb{R}^{d\times n})\}.
\]

Following \cite{Peng2006a, Peng2006b}, we can construct a sublinear
expectation $\mathbb{E}$ on
$(\Omega,L_{ip}(\Omega))$, called $G$-expectation, such that $(B_{t}%
(\omega))_{t\geq0}$ is a $G$-Brownian motion. Since the natural
correspondence of $L_{ip}(\bar{\Omega})$ and $L_{ip}(\Omega)$, we
can also construct a
sublinear expectation $\mathbb{\bar{E}}$ on $(\bar{\Omega}%
,L_{ip}(\bar{\Omega}))$ such that $(\bar{B}_{t}(\bar{\omega}))_{t\geq0}$ is
also a $G$-Brownian motion. In particular, for each $0\leq s<t<\infty$,
$\mathbf{a}\in \mathbb{R}^{d}$ and $p\in \lbrack1,\infty)$,%
\begin{equation}
\mathbb{\bar{E}}[|\left(  \mathbf{a},\bar{B}_{t}-\bar{B}_{s}\right)
|^{p}]=\frac{1}{\sqrt{2\pi \sigma_{\mathbf{aa}^{T}}^{2}(t-s)}}\int_{-\infty
}^{\infty}|x|^{p}\exp \left(  \frac{-x^{2}}{2\sigma_{\mathbf{aa}^{T}}^{2}%
(t-s)}\right)  \,dx,\label{normal}%
\end{equation}
where $\sigma_{\mathbf{aa}^{T}}^{2}=2G(\mathbf{aa}^{T})$.

In \cite{Peng2006a}, \cite{Peng2006b}, \cite{Peng2007b} the space
$L_{ip}(\Omega)$ is extended to $L_{G}^{p}(\Omega)$ under the Banach
norm $\mathbb{E}[|\cdot|]$ to develop a new type of $G$-stochastic
calculus, including $G$-It\^{o}'s integrals, $G$-It\^{o}'s formula
and $G$-SDE. In \cite{DHP} a family of weakly compact probability
measures has been found to represent $\mathbb{E}$. This
representation theorem is essentially important. Indeed, through it
we were able to prove in \cite{DHP} that an element $Y$ of the
abstract Banach space $L_{G}^{p}(\Omega)$ is in fact a
quasi-continuous function $Y=Y(\omega)$ defined on $\Omega$, with
respect to the natural capacity induced by this family. The space
$L_{G}^{p}(\Omega)$ is also proved to be identified with the that
introduced in \cite{denma}. In the next section we give a very
simple and elementary proof of this representation theorem.

\section{$G$-Expectation as an upper Expectation}

A main objective of this paper is to find a weakly compact family of ($\sigma
$-additive) probability measures on $(\Omega,\mathscr{B}(\Omega))$ to
represent $G$-expectation $\mathbb{E}$. We need the following Lemmas.

\begin{lemma}
\label{le3} Let $0\leq t_{1}<t_{2}<\cdots<t_{m}<\infty$ and $\{
\varphi_{n}\}_{n=1}^{\infty} \subset C_{l.Lip}(\mathbb{R}^{d\times
m})$ satisfy $\varphi_{n}\downarrow0$. Then
$\mathbb{\bar{E}}[\varphi_{n}(\bar
{B}_{t_{1}},\bar{B}_{t_{2}},\cdots,\bar{B}_{t_{m}})]\downarrow0$.
\end{lemma}

\begin{proof}
We denote by
$X=(\bar{B}_{t_{1}},\bar{B}_{t_{2}},\cdots,\bar{B}_{t_{m}})$. For
each $N>0$, it is clear that
\[
\varphi_{n}(x)\leq k^{N}_{n}+\varphi_{1}(x)\mathbf{1}_{[|x|>N]} \leq
k^{N}_{n}+\frac{\varphi_{1}(x)|x|}{N} \quad \mbox{for each} \quad
x\in \mathbb{R}^{d\times m},
\]
where $k^{N}_{n}\triangleq \max_{|x|\leq N}\varphi_{n}(x)$. Noting
that $\varphi_{1}(x)|x| \in C_{l.Lip}(\mathbb{R}^{d\times m})$, then
we have
\[
\mathbb{\bar{E}}[\varphi_{n}(X)] \leq
k^{N}_{n}+\frac{1}{N}\mathbb{\bar{E}}[\varphi_{1}(X)|X|].
\]
It follows from $\varphi_{n}\downarrow0$ that
$k^{N}_{n}\downarrow0$. Thus we have $\lim_{n\rightarrow
\infty}\mathbb{\bar{E}}[\varphi_{n}(X)]\leq
\frac{1}{N}\mathbb{\bar{E}}[\varphi_{1}(X)|X|]$. Since $N$ can be
arbitrarily large, we get
$\mathbb{\bar{E}}[\varphi_{n}(X)]\downarrow0$.
\end{proof}

We denote by $\mathscr{T}:=\{
\underline{t}=(t_{1},\ldots,t_{m}):\forall m\in \mathbb{N}, 0\leq
t_{1}<t_{2}<\cdots<t_{m}<\infty \}.$

\begin{lemma}
\label{le4} Let $E$ be a finitely additive linear expectation dominated by $\mathbb{\bar{E}%
}$ on $L_{ip}(\bar{\Omega})$. Then there exists a unique probability
measure $Q$ on $(\bar{\Omega},\mathscr{B}(\bar{\Omega}))$ such that
$E[X]=E_{Q}[X]$ for each $X\in L_{ip}(\bar{\Omega})$.
\end{lemma}

\begin{proof}
For each fixed $\underline{t}=(t_{1},\ldots,t_{m})\in \mathscr{T}$,
by Lemma \ref{le3}, for each sequence $\{
\varphi_{n}\}_{n=1}^{\infty} \subset C_{l.Lip}(\mathbb{R}^{d\times
m})$ satisfying $\varphi_{n}\downarrow0$, we have
$E[\varphi_{n}(\bar{B}_{t_{1}},\bar{B}_{t_{2}},\cdots,\bar{B}_{t_{m}%
})]\downarrow0$. By Daniell-Stone's theorem, there exists a unique
probability
measure $Q_{\underline{t}}$ on $(\mathbb{R}^{d\times m}%
,\mathscr{B}(\mathbb{R}^{d\times m}))$ such that $E_{Q_{\underline{t}}%
}[\varphi]=E[\varphi(\bar{B}_{t_{1}},\bar{B}_{t_{2}},\cdots,\bar{B}_{t_{m}}%
)]$ for each $\varphi \in C_{l.Lip}(\mathbb{R}^{d\times m})$. Thus
we get a family of finite-dimensional distributions
$\{Q_{\underline{t}}:\underline {t}\in \mathscr{T}\}$, by
Daniell-Stone's theorem, it is easy to check that
$\{Q_{\underline{t}}:\underline{t}\in \mathscr{T}\}$ is consistent,
then by Kolmogorov's consistent theorem, there exists a probability
measure $Q$ on
$(\bar{\Omega},\mathscr{B}(\bar{\Omega}))$ such that $\{Q_{\underline{t}%
}:\underline{t}\in \mathscr{T}\}$ is the finite-dimensional
distributions of $Q$. Assume there exists another probability
measure $\bar{Q}$ satisfying the condition, by Daniell-Stone's
theorem, $Q$ and $\bar{Q}$ have the same finite-dimensional
distributions, then by monotone class theorem, $Q=\bar{Q}$. The
proof is complete.
\end{proof}

\begin{lemma}
\label{le5} There exists a family of probability measures $\mathscr{P}_{\!
\!e}$ on $(\bar{\Omega},\mathscr{B}(\bar{\Omega}))$ such that
\[
\mathbb{\bar{E}}[X]=\max_{Q\in \mathscr{P}_{\! \!e}}E_{Q}[X],\quad \forall X\in
L_{ip}(\bar{\Omega}).
\]

\end{lemma}

\begin{proof}
By Lemma \ref{le0} and Lemma \ref{le4}, it is easy to get the
result.
\end{proof}

\  \  \

For this $\mathscr{P}_{\! \!e}$, we define the associated capacity
\[
\tilde{c}(A):=\sup_{Q\in \mathscr{P}_{\! \!e}}Q(A),\quad A\in \mathscr{B}(\bar
{\Omega}).
\]
and upper expectation for each
$\mathscr{B}(\bar{\Omega})$-measurable real function $X$ which makes
the following definition meaningful,
\[
\mathbb{\tilde{E}}[X]:=\sup_{Q\in \mathscr{P}_{\! \!e}}E_{Q}[X].
\]

\begin{lemma}
\label{le6} For $\bar{B}=\{ \bar{B}_{t}:t\in \lbrack0,\infty)\}$ , there exists
a continuous modification $\tilde{B}=\{ \tilde{B}_{t}:t\in \lbrack0,\infty)\}$
of $\bar{B}$ (i.e. $\tilde{c} (\{ \tilde{B}_{t}\not =\bar{B}_{t}\})=0$, for
each $t\geq0$) such that $\tilde{B}_{0}=0$.
\end{lemma}

\begin{proof}
By Lemma \ref{le5}, we know that $\mathbb{\bar{E}}=\mathbb{\tilde{E}}$ on
$L_{ip}(\bar{\Omega})$, from (\ref{normal}) we get
\[
\mathbb{\tilde{E}}[|\bar{B}_{t}-\bar{B}_{s}|^{4}]=\mathbb{\bar{E}}[|\bar{B}%
_{t}-\bar{B}_{s}|^{4}]=d|t-s|^{2},\forall s,t\in \lbrack0,\infty),
\]
where $d$ is a constant depending only on $G$. By generalized
Kolmogorov's criterion for continuous modification with respect to
capacity (see Theorem 31 in [DHP]), there exists a continuous
modification $\tilde{B}$ of $\bar{B}$. Since $\tilde{c}(\{
\bar{B}_{0}\not =0\})=0$, we can set $\tilde{B}_{0}=0$. The proof is
complete.
\end{proof}

\  \  \

For each $Q\in \mathscr{P}_{\! \!e}$, let $Q\circ \tilde{B}^{-1}$ denote the
probability measure on $(\Omega,\mathscr{B}(\Omega))$ induced by $\tilde{B}$
with respect to $Q$. We denote by $\mathscr{P}_{1}=\{Q\circ \tilde{B}^{-1}%
:Q\in \mathscr{P}_{\! \!e}\}$. By Lemma \ref{le6}, we get
\[
\mathbb{\tilde{E}}[|\tilde{B}_{t}-\tilde{B}_{s}|^{4}]=\mathbb{\tilde{E}}[|\bar{B}_{t}-\bar{B}_{s}|^{4}]=
d|t-s|^{2},\forall s,t\in \lbrack0,\infty).
\]
Applying the well-known result of moment criterion for tightness of
Kolmogorov-Chentrov's type, we conclude that $\mathscr{P}_{1}$ is
tight. We denote by $\mathscr{P}=\overline{\mathscr{P}}_{1}$ the
closure of $\mathscr{P}_{1}$ under the topology of weak convergence,
then $\mathscr{P}$ is weakly compact.

Now, we give the representation of $G$-expectation.

\begin{theorem}\label{th7} For each continuous monotonic and sublinear function $G:\mathbb{S}_d\mapsto \mathbb{R}$, let $\mathbb{E}$ be the corresponding
$G$-expectation on $(\Omega,L_{ip}(\Omega))$. Then there exists a
weakly compact family of probability measures $\mathscr{P}$ on
$(\Omega,\mathscr{B}(\Omega))$ such that

\[
\mathbb{E}[X]=\max_{P\in \mathscr{P}}E_{P}[X],\quad \forall X \in L_{ip}%
(\Omega).
\]

\end{theorem}

\begin{proof}
By Lemma \ref{le5} and Lemma \ref{le6}, we have
\[
\mathbb{E}[X]=\max_{P\in \mathscr{P}_{1}}E_{P}[X],\quad \forall X\in
L_{ip}(\Omega).
\]
For each $X\in L_{ip}(\Omega)$, by Lemma \ref{le3}, we get
$\mathbb{E}[|X-(X\wedge N)\vee (-N)|]\downarrow 0$ as $N\rightarrow
\infty$. Noting also that $\mathscr{P}=\overline{\mathscr{P}}_{1}$,
then by the definition of weak convergence, we get the result.
\end{proof}

\section{Completion of $L_{ip}(\Omega)$}

We denote by $L^{0}(\Omega)$ the space of all
$\mathscr{B}(\Omega)$-measurable real functions and $C_{b}(\Omega)$
all bounded continuous functions. In section 3, we obtain a weakly
compact family $\mathscr{P}$ of probability measures on
$(\Omega,\mathscr{B}(\Omega))$ to represent $G$-expectation
$\mathbb{E}$. For this $\mathscr{P}$, we define the associated
capacity
\[
\hat{c}(A):=\sup_{P\in \mathscr{P}}P(A),\quad A\in \mathscr{B}(\Omega).
\]
and upper expectation for each $X\in L^{0}(\Omega)$ which makes the
following definition meaningful,
\[
\mathbb{\hat{E}}[X]:=\sup_{P\in \mathscr{P}}E_{P}[X].
\]
By Theorem \ref{th7}, we know that $\mathbb{\hat{E}}=\mathbb{E}$ on
$L_{ip}(\Omega)$, thus the $\mathbb{E}[|\cdot|]$-completion and the
$\mathbb{\hat{E}}[|\cdot|]$-completion of $L_{ip}(\Omega)$ are the same. We
also denote, for $p>0$,

\begin{itemize}
\item $\mathscr{L}^{p}:=\{X\in L^{0}(\Omega):\mathbb{\hat{E}}[|X|^{p}%
]=\sup_{P\in \mathscr{P}}E_{P}[|X|^{p}]<\infty \}$;\

\item $\mathscr{N}^{p}:=\{X\in L^{0}(\Omega):\mathbb{\hat{E}}[|X|^{p}]=0\}$;

\item $\mathscr{N}:=\{X\in L^{0}(\Omega):X=0$, $\hat{c}$-q.s.$\}$.
\end{itemize}

It is seen that $\mathscr{L}^{p}$ and $\mathscr{N}^{p}$ are linear spaces and
$\mathscr{N}^{p}=\mathscr{N}$, for each $p>0$.

We denote by $\mathbb{L}^{p}:=\mathscr{L}^{p}/\mathscr{N}$. As
usual, we do not take care about the distinction between classes and
their representatives.

Now, we give the following two Propositions which can be found in [DHP].

\begin{proposition}
\label{pr8} For each $\{X_{n} \}_{n=1}^{\infty}$ in $C_{b}(\Omega)$ such that
$X_{n}\downarrow0$ \ on $\Omega$, we have $\mathbb{\hat{E}} [X_{n}%
]\downarrow0$.
\end{proposition}

\begin{proposition}
\label{pr9} We have

\begin{enumerate}
\item For each $p\geq1$, $\mathbb{L}^{p}$ is a Banach space under the norm
$\left \Vert X\right \Vert _{p}:=\left(  \mathbb{\hat{E} }[|X|^{p}]\right)
^{\frac{1}{p}}$.

\item For each $p<1$, $\mathbb{L}^{p}$ is a complete metric space under the
distance \newline$d(X,Y):= \mathbb{\hat{E} }[|X-Y|^{p}] $.
\end{enumerate}
\end{proposition}

With respect to the distance defined on $\mathbb{L}^{p}$, $p>0$, we denote:

\begin{itemize}
\item $\mathbb{L}_{c}^{p}$ the completion of $C_{b}(\Omega)$.

\item $L_{G}^{p}(\Omega)$ the completion of $L_{ip}(\Omega)$.
\end{itemize}

For each $T>0$, we also denote by $\Omega_{T}=C_{0}^{d}([0,T])$
equipped with the distance
\[
\rho(\omega^{1},\omega^{2})=\left \Vert \omega^{1}-\omega^{2} \right \Vert
_{C_{0}^{d}([0,T])}:=\max_{0\leq t\leq T}|\omega^{1}_{t}-\omega^{2}_{t}|.
\]

We now prove that $L_{G}^{1}(\Omega)=\mathbb{L}_{c}^{1}$. First, we need the
following classical approximation Lemma.

\begin{lemma}
\label{le10} For each $X\in C_{b}(\Omega)$ and $n=1,2,\cdots$, we denote
\[
X^{(n)}(\omega)\triangleq \inf_{\omega^{\prime}\in \Omega}\{X(\omega^{\prime
})+n\left \Vert \omega-\omega^{\prime}\right \Vert _{C_{0}^{d}([0,n])}\},
\quad \forall \omega \in \Omega.
\]
Then the sequence $\{X^{(n)}\}_{n=1}^{\infty}$ satisfies:

\begin{enumerate}
\item $-M\leq X^{(n)}\leq X^{(n+1)}\leq \cdots \leq X$, $M=\sup_{\omega \in
\Omega}|X(\omega)|.$\

\item $|X^{(n)}(\omega_{1})-X^{(n)}(\omega_{2})|\leq n\left \Vert \omega
_{1}-\omega_{2}\right \Vert _{C_{0}^{d}([0,n])},\  \  \forall \omega_{1}%
,\omega_{2}\in \Omega.$

\item $X^{(n)}(\omega)\uparrow X(\omega),\  \  \forall \omega \in \Omega.$
\end{enumerate}


\end{lemma}

\begin{proof}
\textbf{1.}is obvious.

For \textbf{2.} We have\textbf{ }
\[%
\begin{array}
[c]{l}%
\displaystyle X^{(n)}(\omega_{1})-X^{(n)}(\omega_{2})\\
\displaystyle \qquad \leq \sup_{\omega^{\prime}\in \Omega}\{[X(\omega^{\prime
})+n\left \Vert \omega_{1}-\omega^{\prime}\right \Vert _{C_{0}^{d}%
([0,n])}]-[X(\omega^{\prime})+n\left \Vert \omega_{2}-\omega^{\prime
}\right \Vert _{C_{0}^{d}([0,n])}]\} \\
\displaystyle \qquad \leq n\left \Vert \omega_{1}-\omega_{2}\right \Vert
_{C_{0}^{d}([0,n])}%
\end{array}
\]
and, symmetrically, $X^{(n)}(\omega_{2})-X^{(n)}(\omega_{1})\leq n\left \Vert
\omega_{1}-\omega_{2}\right \Vert _{C_{0}^{d}([0,n])}$. Thus \textbf{2} follows.

We now prove \textbf{3}.
For each fixed $\omega \in \Omega$, let $\omega_{n}\in \Omega$ be such that
\[
X(\omega_{n})+n\left \Vert \omega-\omega_{n}\right \Vert _{C_{0}^{d}([0,n])}\leq
X^{(n)}(\omega)+\frac{1}{n}.
\]
It is clear that $n\left \Vert \omega-\omega_{n}\right \Vert _{C_{0}^{d}%
([0,n])}\leq2M+1$, or $\left \Vert \omega-\omega_{n}\right \Vert _{C_{0}%
^{d}([0,n])}\leq \frac{2M+1}{n}$. Since $X\in C_{b}(\Omega)$, we get
$X(\omega_{n})\rightarrow X(\omega)$ as $n\rightarrow \infty$. We have
\[
X(\omega)\geq X^{(n)}(\omega)\geq X(\omega_{n})+n\left \Vert \omega-\omega
_{n}\right \Vert _{C_{0}^{d}([0,n])}-\frac{1}{n},
\]
thus%
\[
n\left \Vert \omega-\omega_{n}\right \Vert _{C_{0}^{d}([0,n])}\leq
|X(\omega)-X(\omega_{n})|+\frac{1}{n}.
\]
We also have%
\begin{align*}
X(\omega_{n})-X(\omega)+n\left \Vert \omega-\omega_{n }\right \Vert _{C_{0}%
^{d}([0,n])}  &  \geq X^{(n)}(\omega)-X(\omega)\\
&  \geq X(\omega_{n})-X(\omega)+n\left \Vert \omega-\omega_{n}\right \Vert
_{C_{0}^{d}([0,n])}-\frac{1}{n}.
\end{align*}
From the above two relations we obtain%
\begin{align*}
|X^{(n)}(\omega)-X(\omega)|  &  \leq|X(\omega_{n})-X(\omega)|+n\left \Vert
\omega-\omega_{n}\right \Vert _{C_{0}^{d}([0,n])}+\frac{1}{n}\\
&  \leq2(|X(\omega_{n})-X(\omega)|+\frac{1}{n})\rightarrow0 \  \text{as}
\ n\rightarrow \infty.
\end{align*}
Thus \textbf{3} is obtained.
\end{proof}

\begin{proposition}
\label{pr11} For each $X\in C_{b}(\Omega)$ and $\varepsilon>0$ there exists a
$Y\in L_{ip}(\Omega)$ such that $\mathbb{\hat{E}}[|Y-X|]\leq \varepsilon$.
\end{proposition}

\begin{proof}
We denote by $M=\sup_{\omega \in \Omega}|X(\omega)|$. By Proposition
\ref{pr8} and Lemma \ref{le10}, we can find $\mu>0$, $T>0$ and
$\bar{X}\in C_{b}(\Omega _{T})$ such that
$\mathbb{\hat{E}}[|X-\bar{X}|]<\varepsilon/3$, $\sup _{\omega \in
\Omega}|\bar{X}(\omega)|\leq M$ and
\[
|\bar{X}(\omega)-\bar{X}(\omega^{\prime})|\leq \mu \left \Vert \omega
-\omega^{\prime}\right \Vert _{C_{0}^{d}([0,T])},\  \  \forall \omega
,\omega^{\prime}\in \Omega.
\]
Now for each positive integer $n$, we introduce a mapping:$\omega^{(n)}%
(\omega): \Omega \mapsto \Omega$ by%
\[
\omega^{(n)}(\omega)(t)=\sum_{k=0}^{n-1}\frac{\mathbf{1}_{[t_{k}^{n}%
,t_{k+1}^{n})}(t)}{t_{k+1}^{n}-t_{k}^{n}}[(t_{k+1}^{n}-t)\omega(t_{k}%
^{n})+(t-t_{k}^{n})\omega(t_{k+1}^{n})]+ \mathbf{1}_{[T,\infty)}%
(t)\omega(t),\
\]
where $t_{k}^{n}=\frac{kT}{n},\ k=0,1,\cdots,n$. We set $\bar{X}^{(n)}%
(\omega):=\bar{X}(\omega^{(n)}(\omega))$, then
\begin{align*}
|\bar{X}^{(n)}(\omega)-\bar{X}^{(n)}(\omega^{\prime})|  &  \leq \mu \sup
_{t\in \lbrack0,T]}|\omega^{(n)}(\omega)(t)-\omega^{(n)}(\omega^{\prime})(t)|\\
&  =\mu \sup_{k\in \lbrack0,\cdots,n]}|\omega(t_{k}^{n})-\omega^{\prime}%
(t_{k}^{n})|.
\end{align*}
We now choose a compact subset $K\subset \Omega$ such that
$\mathbb{\hat{E}}[\mathbf{1}_{K^{C}}]\leq \varepsilon/6M$. Since
$\sup_{\omega \in K}\sup _{t\in
\lbrack0,T]}|\omega(t)-\omega^{(n)}(\omega)(t)|\rightarrow0$, as
$n\rightarrow \infty$, we then can choose a sufficiently large
$n_{0}$ such that
\begin{align*}
\sup_{\omega \in K}|\bar{X}(\omega)-\bar{X}^{(n_{0})}(\omega)|  &
=\sup_{\omega \in K}|\bar{X}(\omega)-\bar{X}(\omega^{(n_{0})}(\omega))|\\
&  \leq \mu \sup_{\omega \in K}\sup_{t\in \lbrack0,T]}|\omega(t)-\omega^{(n_{0}%
)}(\omega)(t)|\\
&  <\varepsilon/3.
\end{align*}
We set $Y:=\bar{X}^{(n_{0})}$, it follows that
\begin{align*}
\mathbb{\hat{E}}[|X-Y|]  &  \leq \mathbb{\hat{E}}[|X-\bar{X}|]+\mathbb{\hat{E}%
}[|\bar{X}-\bar{X}^{(n_{0})}|]\\
&  \leq \mathbb{\hat{E}}[|X-\bar{X}|]+\mathbb{\hat{E}}[\mathbf{1}_{K}|\bar
{X}-\bar{X}^{(n_{0})}|]+2M\mathbb{\hat{E}}[\mathbf{1}_{K^{C}}]\\
&  <\varepsilon.
\end{align*}
The proof is complete.
\end{proof}

\  \  \

By Proposition \ref{pr11}, we can easily get $L_{G}^{1}(\Omega)=\mathbb{L}%
_{c}^{1}$. Furthermore, we can get $L_{G}^{p}(\Omega)=\mathbb{L}_{c}^{p}$,
$\forall p>0$.

\end{document}